\documentclass{amsart}
\usepackage{amssymb}
\usepackage{mathrsfs}
\makeatletter

\def\<#1\>{\left\langle#1\right\rangle}

\let\Re\undefined
\DeclareMathOperator{\Re}{Re}

\DeclareMathOperator{\Ind}{Ind}
\DeclareMathOperator{\Pe}{P\acute e}

\DeclareMathOperator{\ch}{ch}
\DeclareMathOperator{\id}{id}
\DeclareMathOperator{\vol}{vol}

\newcommand{\dd}[1]{\mathop{\mathrm{d}#1}}

\newcommand{\al}{\alpha}
\newcommand{\ga}{\gamma}
\newcommand{\de}{\delta}
\newcommand{\ep}{\varepsilon}
\newcommand{\la}{\lambda}
\newcommand{\ph}{\varphi}
\newcommand{\Th}{\varTheta}
\renewcommand{\Phi}{\varPhi}
\renewcommand{\Psi}{\varPsi}
\newcommand{\Cas}{\Omega}
\newcommand{\Gam}{\Gamma}
\newcommand{\Lap}{\Delta}

\def\abs{\@ifstar{\abs@star}{\abs@nostar}}
\newcommand{\abs@nostar}[1]{\lvert#1\rvert}
\newcommand{\abs@star}[1]{\left\lvert#1\right\rvert}
\def\Abs{\@ifstar{\Abs@star}{\Abs@nostar}}
\newcommand{\Abs@nostar}[1]{\lVert#1\rVert}
\newcommand{\Abs@star}[1]{\left\lVert#1\right\rVert}

\newcommand{\ts}{^{\mathrm{t}}}
\newcommand{\const}{\mathrm{c}}

\newcommand{\cj}[1]{{\overline{#1}}}

\newcommand{\ade}{\mathbb{A}}

\newcommand{\lqu}{\backslash}

\newcommand{\makegroup}[7][k]{%
  \newcommand{#3}[1][#1]{{#2_{##1}}}%
  \newcommand{#4}{{#2_\ade}}%
  \newcommand{#5}[1][#1]{{#3[##1]\lqu#4}}%
  \newcommand{#6}[1][v]{{#2_{##1}}}
  \newcommand{#7}{{#2_\infty}}}

\makegroup{G}{\Gk}{\Ga}{\Gq}{\Gv}{\Gin}
\makegroup{H}{\Hk}{\Ha}{\Hq}{\Hv}{\Hin}
\makegroup{M}{\Mk}{\Ma}{\Mq}{\Mv}{\Min}
\makegroup{N}{\Nk}{\Na}{\Nq}{\Nv}{\Nin}
\makegroup{P}{\Pk}{\Pa}{\Pq}{\Pv}{\Pin}
\makegroup{\Th}{\Sk}{\Sa}{\Sq}{\Sv}{\Sin}

\newcommand{\Kv}{{K_v}}
\newcommand{\Kin}{{K_\infty}}

\newcommand{\GL}{\mathrm{GL}}
\newcommand{\gO}{\mathrm{O}}

\newcommand{\sint}{\int^\oplus}

\let\frk\mathfrak
\let\scr\mathscr
\let\cal\mathcal

\newcommand{\ScD}{\scr D}

\newcommand{\Zg}[1][]{{\scr Z(\frk g#1)}}

\newcommand{\CC}{\mathbb{C}}
\newcommand{\RR}{\mathbb{R}}
\newcommand{\QQ}{\mathbb{Q}}

\newcommand{\HCc}{\mathbf{c}}
\newcommand{\spht}{\widetilde}
\newcommand{\Ft}{\widehat}

\newcommand{\Hauto}[1]{{H_{\textup{auto}}^{#1}}}
\newcommand{\Hzon}[1]{{H_{\textup{zonal}}^{#1}}}

\renewcommand{\thmhead}[3]{\thmnumber{{\rm(#2)}\@ifnotempty{#1#3}%
    { }}\thmname{#1\@ifnotempty{#3}{\,---\,}}\thmnote{#3}}
\newcommand{\thref}{\eqref}

\newtheorem{thm}[equation]{Theorem}
\newtheorem{prop}[equation]{Proposition}
\newtheorem{lem}[equation]{Lemma}
\newtheorem*{thm*}{Theorem}
\newtheorem*{prop*}{Proposition}
\newtheorem*{lem*}{Lemma}
\theoremstyle{definition}

\newtheorem*{defn*}{Definition}

\usepackage[lite]{amsrefs}

\begin{document}

\title[A spectral identity for Eisenstein series of $\gO(n, 1)$]
  {A spectral identity for second moments of Eisenstein series of $\gO(n, 1)$}
\author{Jo\~ao Pedro Boavida}
\address{Departamento de Matem\'atica\\
         Instituto Superior T\'ecnico\\Universidade de Lisboa\\
         Av.\ Prof.\ Dr.\ Cavaco Silva\\2744--016 Porto Salvo, Portugal}
\email{joao.boavida@tecnico.ulisboa.pt}
\thanks{Published in Illinois J. Math.\ \textbf{57} (2013), no.\ 4,
  1111\ndash 1130. Final submitted version, other than some citation updates.}
\keywords{Eisenstein series, second moments, period, automorphic, fundamental
  solution, Poincar\'e series, $L$-function, orthogonal group}
\subjclass[2010]{Primary 11F67; Secondary 11E45, 11F72, 43A85, 43A90, 46E35}

\begin{abstract}
Let $H = \gO(n) \times \gO(1)$ be an anisotropic subgroup of $G = \gO(n, 1)$ 
and let $\ade$ be the adele ring of $k = \QQ$.  Consider the periods
\begin{equation*}
  (E_\ph, F)_H
  = \int_{\Hq} E_\ph \cdot \cj F,
\end{equation*}
of an Eisenstein series $E_\ph$ on $G$ against a form $F$ on $H$.  Relying on
a variant of Levi--Sobolev spaces, we describe certain Poincar\'e series as 
fundamental solutions for the laplacian, and use them to establish a spectral 
identity concerning the second moments (in $F$-aspect) of $E_\ph$.
\end{abstract}

\maketitle

\section*{Introduction}

Let $k = \QQ$.  Consider the form represented by
\begin{equation*}
  \begin{pmatrix} 1 & & \\ & \id & \\ & &  -1 \end{pmatrix}
\end{equation*}
(here and elsewhere, omitted entries are zero) with respect to the 
decomposition $k^{n+1} = k^n \oplus (k \cdot e_-)
= (k \cdot e_+) \oplus k^{n-1} \oplus (k \cdot e_-)$.  
Let $G = \gO(n+1)$, $H = \gO(n) \times \gO(1)$, and $\Th = \gO(n-1)$, and 
note that $H$ and $\Th$ are $k$-anisotropic.

As the form is isotropic, we consider the hyperbolic pair 
$e' = \frac12 e_+ - \frac12 e_-$ and $e = e_+ + e_-$.  Changing coordinates, 
we see the form is represented by
\begin{equation*}
  \begin{pmatrix} & & 1 \\ & \id & \\ 1 & & \end{pmatrix}
\end{equation*}
with respect to $(k\cdot e') \oplus k^{n-1} \oplus (k\cdot e)$.  We use these
new coordinates for the remainder of the introduction, and observe that
while $H$ has no simple description in these coordinates, $\Th$ can still be
identified with $\gO(n-1)$.

Write
\begin{equation*}
  m_\la
  = \begin{pmatrix} \la & & \\ & \id & \\ & & \la^{-1} \end{pmatrix}
  \qquad\text{and}\qquad
  n_a
  = \begin{pmatrix} 1 & a & -\frac12aa\ts \\
                    & \id & -a\ts \\ & & 1 \end{pmatrix}.
\end{equation*}
The parabolic $P$ stabilizing the isotropic line $k\cdot e$ can be written
as $P = NM$, with unipotent radical $N = \{n_a\}$ and Levi component 
$M = \{m_\la\} \cdot \Th$.  The modular function on $P$ is given by
$\de_P(m_\la) = \abs\la^n$.

Let $\ade$ be the adele ring of $k = \QQ$.  At non-archimedean $v$, choose a 
maximal (open) compact $\Kv$.  At the archimedean place $v = \infty$, put 
$\Kin = \Hin$; it is a maximal compact in $\Gin$.  Write $K = \prod \Kv$; it 
is a maximal compact in $\Ga$.  Let us recapitulate briefly the most salient 
points about the spectral decomposition of (right $K$--invariant) functions in 
$L^2(\Gq/K)$.

The \emph{constant term} of $f \in L^2(\Gq/K)$ is
\begin{equation*}
  \const f(g)
  = \int_{\Nq} f(ng) \dd n.
\end{equation*}
We say $f$ is a \emph{cuspform} if $\const f = 0$; the space $L^2_0(\Gq/K)$ of 
(right $K$--invariant) cuspforms decomposes discretely \cite{God} into joint 
eigenfunctions of the center $\Zg[_\infty]$ of the universal enveloping 
algebra.

The constant term $\const f$ is left $\Na\Mk$--invariant.  If 
$\ph \in \ScD(\Na\Mk\lqu\Ga/K)$ is a test function, we have
\begin{equation*}\begin{split}
  \int_{\Na\Mk\lqu\Ga/K} \const f(g) \, \ph(g) \dd g
  &= \int_{\Na\Mk\lqu\Ga/K} \int_{\Nq} f(ng) \dd n \, \ph(g) \dd g \\
  &= \int_{\Pk\lqu\Ga/K} f(g) \, \ph(g) \dd g\\
  &= \int_{\Gq/K} f(g) \, E_\ph(g) \dd g,
\end{split}\end{equation*}
where
\begin{equation*}
  E_\ph(g)
  = \sum_{\ga \in \Pk\lqu\Gk} \ph(\ga g)
\end{equation*}
(the sum has finitely many nonzero terms) is a 
\emph{pseudo-Eisenstein series}.  Observing that $\Na\Mk = \Mk\Na$ and taking
the Iwasawa decomposition $\Ga = \Na\Ma K$ into account, we see  the right
$K$--invariant functions on $\Na\Mk\lqu\Ga$ are the right $K\cap\Ma$--invariant
functions on $\Mq$.

Recall that $M \cong \Th \times \GL(1)$ and that, because
$\Th$ is $k$-anisotropic, $\Sq$ is compact.  Let $\Psi$ run over an orthonormal
basis of $L^2(\Sq/(K\cap\Sa))$.  Let also $\la \mapsto \de_P(\la)^s$ be a
character of $\GL(1)$ (with $k = \QQ$, there are no other characters to 
account for).  Extend
\begin{equation*}
  \ph_{s,\Psi}(m_\la \theta)
  = \de_P(\la)^s \cdot \Psi(\theta)
  = \abs\la^{ns} \cdot \Psi(\theta)
\end{equation*}
by left $\Na$-- and right $K$--invariance, and define the \emph{Eisenstein
series} as the meromorphic continuation of
\begin{equation*}
  E_{s,\Psi} (g)
  = \sum_{\ga \in \Pk\lqu\Gk} \ph_{s,\Psi}(\ga g).
\end{equation*}
to all $\CC$.  (We will not go into the details, but the sum converges if 
$\Re s > 1$ and does have a meromorphic extension \cite{Lan}.)  If $\Psi = 1$, 
we write simply $\ph_s = \ph_{s,\Psi}$ and $E_s = E_{s,\Psi}$.

Given a function $f$ in $L^2(\Gq/K)$, we have \cites{Ar, Lan, MoWa}
\begin{equation*}
  f
  = \sum_{\Phi} \<f, \Phi\> \cdot \Phi
    + \frac1{4\pi i}\sum_\Psi
      \int_{\Re s = \frac12}\<f, E_{s,\Psi}\> \cdot E_{s,\Psi} \dd s
    + \sum_R \<f, R\> \cdot R
\end{equation*}
with $\Phi$ running over an orthonormal basis of $L^2_0(\Gq/K)$, $\Psi$ over an 
orthonormal basis of $L^2(\Sq/(K\cap\Sa))$, and $R$ over an orthonormal basis 
of residues of Eisenstein series to the right of $\Re s = \frac12$.  (The inner
products are integrals over $\Gq$.)

We choose 
each component to be a joint eigenvector of $\Zg[_\infty]$.  The corresponding 
Plancherel identity is
\begin{equation*}
  \Abs f^2_{L^2}
  = \sum_\Phi\abs*{\<f, \Phi\>}^2
    + \frac1{4\pi i} \sum_\Psi
      \int_{\Re s = \frac12}\abs*{\<f, E_{s,\Psi}\>}^2\dd s
    + \sum_R\abs*{\<f, R\>}^2
\end{equation*}
(We note that the Eisenstein series themselves are not in $L^2$, therefore the 
inner product and integral are obtained  by isometric extension.)

In what follows, we shorten these formulas to read
\begin{equation}\label{e:spec}
  f
  = \sint \<f, \Phi\> \cdot \Phi \dd\Phi
  \qquad\text{and}\qquad
  \Abs f^2_{L^2}
  = \sint\abs*{\<f, \Phi\>}^2 \dd\Phi
\end{equation}
(when writing thusly, $\Phi$ runs over all relevant spectral components).

We may consider the periods
\begin{equation*}
  (\Phi, F)_H
  = \int_{\Hq} \Phi \cdot \cj F
\end{equation*}
of spectral components $\Phi$ on $G$ against cuspforms $F$ on $H$, or even
\begin{equation*}
  (\Phi)_H
  = (\Phi, 1)_H
  = \int_{\Hq} \Phi.
\end{equation*}
Such periods contain information about the underlying representations.  These 
same periods (called there global Shintani functions) were used by Katu,
Murase, and Sugano \cites{KaMuSu, MuSu} to obtain and study integral expressions 
for standard $L$-functions of the orthogonal group.  And the Gross--Prasad 
conjecture \cites{GrPr1, GrPr2, GrPr3} predicts that a representation of 
$\gO(n)$ occurs in a representation of $\gO(n+1)$ if and only if the 
corresponding tensor product $L$-function is nonzero on $\Re s = \frac12$.  
Ichino and Ikeda \cite{IcIk} discuss further details and 
broader context is provided in papers by Gross, Reeder \cite{GrRe}, Jacquet, 
Lapid, Offen, and/or Rogawski \cites{JaLaRo, LaRo, LaOf}, Jiang \cite{Ji},
and Sakellaridis and Venkatesh \cites{Sak, SaVe}.

The periods also help study the asymptotics of moments of automorphic
$L$-functions.  Often, the Phragm\'en--Lindel\"of principle yields (so-called) 
\emph{convex} bounds for such asymptotics \cites{BeRe3, IwSa}.  Diaconu and Garrett 
\cites{DiGa1, DiGa2} used a specific spectral identity to first break convexity 
for the asymptotics of second moments of automorphic forms in $\GL(2)$, over 
\emph{any} number field $k$.  In fact, their strategy produces families of 
spectral identities, explored in other papers by them and/or Goldfeld 
\cites{DiGa2, DiGa3, DiGaGo} and used by Letang \cite{Le}.
In the present paper, we carry out that strategy to obtain a spectral identity
for second moments of Eisenstein series of $\gO(n, 1)$.

Given a function $f \in L^2(\Gq/K)$, the spectral decomposition \eqref{e:spec} 
above invites us to consider the effect of an operator $X \in \Zg[_\infty]$:
\begin{equation}\label{e:eigen}
  Xf
  = \sint \<f, \Phi\> \cdot \la_{X,\Phi} \cdot \Phi \dd\Phi
  \qquad\text{and}\qquad
  \Abs{Xf}^2_{L^2}
  = \sint\abs*{\<f, \Phi\>}^2 \, \abs{\la_{X,\Phi}}^2 \dd\Phi,
\end{equation}
where $\la_{X,\Phi}$ is the $X$-eigenvalue of $\Phi$ (if $X = \Cas$, 
we write simply $\la_\Phi=\la_{\Cas,\Phi}$).  The 
conditions for these decompositions to converge (even in the sense of isometric 
extensions) are most naturally discussed in the context of automorphic Sobolev 
spaces.  The literature on \emph{automorphic} Sobolev spaces is scarce; it 
includes papers by Bernstein and Reznikov \cites{BeRe1, BeRe2}, Kr\"otz and 
Stanton \cite{KrSt}, and Michel and Venkatesh \cite{MiVe}, as well as 
Garrett's \cite{Ga3} notes and DeCelles's \cite{De} very detailed discussion.  
We discuss them (and 
their zonal counterparts) in sections \ref{s:auto-Sob} and \ref{s:zonal-Sob},
following the approach in the author's dissertation \cite{Bo1}.

The \emph{automorphic Sobolev spaces} we discuss in section \ref{s:auto-Sob} 
are closures (with respect to the relevant norms) of the space $\ScD(\Gq)$ of 
\emph{global} test functions. Even though we only take into account the 
eigenvalues of $\Cas$ in their definition, we rely on a global spectral 
decomposition, and the norms are defined from integrals over $\Gq$. So we 
should see these spaces as spaces of global functions.

A crucial point is that, using a pre-trace kernel, we can obtain an
estimate
\begin{equation*}
  \sint_{\abs{\la_\Phi} < T^2} \abs{\Phi(g)}^2 \ll T^n
\end{equation*}
similar to Weyl's Law, from which we can characterize an automorphic
delta $\de_\ade$.  Then, it is just a matter of using the techniques one
habitually uses with classical Sobolev spaces to obtain fundamental solutions 
of PDEs.

By contrast, the \emph{zonal Sobolev spaces} we discuss in section 
\ref{s:zonal-Sob} are closures of test functions on $\Kin\lqu\Gin/\Kin$; these
are local (archimedean) functions. From them, we shall obtain a different
construction of the (global) fundamental solutions just mentioned, which will
help us extract some archimedean information.

In section \ref{s:Pe}, we use those techniques to obtain fundamental 
solutions (following Diaconu and Garrett \cite{DiGa1}, we call them 
\emph{Poincar\'e series}) for certain polynomials
in $\Cas$.  The spectral decomposition of these Poincar\'e series $\Pe$ 
involves the periods $(\Phi)_H$ discussed above.  Given an automorphic function
$f \otimes f'$ on $G \times G$, we expand $\<f \cdot f', \Pe\>_G$ in two
distinct ways, yielding an identity between a spectral expansion (along $G$)
and a moment expansion (in $F$-aspect, with $F$ running over an orthonormal
basis of cuspforms on $H$).

In section \ref{s:Eis}, we apply those ideas to Eisenstein series.  In 
particular, we see how the moment expansion involves the second moments of 
the Eisenstein series in $F$-aspect, as well as the periods of Eisenstein
series.  (Elsewhere \cites{Bo1, Bo2, Bo3}, this author has computed these 
periods at non-archimedean primes.  As discussed there, for the cases used in 
the present paper, the local factor at the archimedean place is $1$.)

In appendix \ref{s:regularize}, we explain the regularization used in section
\ref{s:Eis}.

\tableofcontents

\subsection*{Acknowledgements}

This paper is based on part of the author's doctoral dissertation \cite{Bo1}, 
done under the supervision of Paul Garrett, and might well not exist without
his continued encouragement since then.  It is influenced by discussions 
with and talks by him \cite{Ga2}, as well as multiple class discussions with 
Adrian Diaconu about their joint work~\cites{DiGa1, DiGa2, DiGa3, Ga25}.  The
author thanks also the referee for a few corrections and suggestions.

\section{Automorphic Sobolev spaces}\label{s:auto-Sob}

In the continuation, we will rely heavily on some $L^2$ Sobolev spaces, 
adapted to the automorphic case.
Classically, the Sobolev space of order $\ell$ is defined as 
the space of functions whose weak derivatives up to order $\ell$ are 
square-integrable.  The topology induced by that family of seminorms (one
for each derivative up to order $\ell$) can also be described by a norm 
obtained from Plancherel formula.  For example, in $\RR^n$, we set
\begin{equation*}
  \Abs{f}^2_{H^\ell}
  = \int_{\RR^n} \abs{\Ft f(\xi)}^2 \, (1+\abs\xi^2)^\ell \dd\xi.
\end{equation*}
Under Fourier transform, the Laplacian $\Lap$ acts (up to a constant) by
multiplication by $\abs\xi^2$.  
In the Plancherel identity, the effect of $\Lap$
is as described in~\eqref{e:eigen}.

In our case, the effect of the Casimir element $\Cas$ of 
$\Gin$ on the Plancherel identity is also as in \eqref{e:eigen}.
Thus, with inner products obtained from integrals over $\Gq$ (or by
isometric extension), we define the automorphic Sobolev norm by
\begin{equation*}
  \Abs{f}^2_\ell
  = \sint\abs{\<f, \Phi\>}^2 \, (1 + \abs{\la_\Phi})^\ell \dd\Phi
\end{equation*}
and the \emph{automorphic Sobolev space} as
\begin{equation}
  \Hauto\ell
  = \,\text{closure of $\ScD(\Gq)$ with respect to $\Abs\;_\ell$.}
\end{equation}
We are specifically interested in the effect of the center $\Zg[_\infty]$ of 
the universal enveloping algebra (and the corresponding differential operators),
so the only modification to the usual $L^2$ norm involves only archimedean
information. However, the norm itself depends on the global automorphic 
spectral decomposition.

For $\ell > 0$, as usual, $\Hauto{-\ell}$ is the dual of $\Hauto\ell$.
Let $f \in \ScD(\Gq) \cap \Hauto{-\ell}$ and
$\ph \in \ScD(\Gq) \cap \Hauto\ell$.  In the expanded notation,
we define $\<f, \ph\>$ by
\begin{equation*}
  \sum_\Phi \<f, \Phi\> \cj{\<\ph, \Phi\>}
    + \frac1{4\pi i} \sum_\Psi \int_{\Re s=\frac12} 
      \<f, E_{s,\Psi}\> \cj{\<\ph, E_{s,\Psi}\>} \dd s
    + \sum_R \<f,R\>\cj{\<\ph, R\>}.
\end{equation*}
From Cauchy--Schwarz--Bunyakowsky, we obtain (now in the compressed 
notation)
\begin{equation*}\begin{split}
  \<f,\ph\>
  &=\sint \<f, \Phi\> \cj{\<\ph, \Phi\>} \dd\Phi\\
  &\ll \sint \abs{\<f, \Phi\>} \; (1 + \abs{\la_\Phi})^{-\ell/2}
      \cdot \abs{\<\ph, \Phi\>} \; (1 + \abs{\la_\Phi})^{\ell/2} \dd\Phi\\
  &\ll \sqrt{\sint\abs{\<f, \Phi\>}^2 \, (1 + \abs{\la_\Phi})^{-\ell} \dd\Phi}
      \cdot\sqrt{\sint\abs{\<\ph, \Phi\>}^2
        \, (1 + \abs{\la_\Phi})^\ell \dd\Phi}\\
  &= \Abs{f}_{-\ell}\cdot\Abs{\ph}_\ell.
\end{split}\end{equation*}

\begin{prop}
With $X = \Gin/\Kin$ and $n = \dim_\RR X$, we have
\begin{equation*}
  \sint_{\abs{\la_\Phi}<T^2}\abs{\Phi(g)}^2
  \ll T^n.
\end{equation*}
\end{prop}
(This is unsurprising, in light of Weyl's Law \cites{Do, Iw, LiVe}.)

\begin{proof}
We follow Garrett \cite{Ga2}, with adjustments to a different group and
attempting to avoid tedious computations.

We use the ``ball''
\begin{equation*}
  B
  = \Bigl\{n_a \, m_\la : \max \abs{a_i} < \frac1T
         \quad\text{and}\quad \abs{\log\la} < \frac1T\Bigr\}
\end{equation*}
of radius $1/T$ in $\Pin$.  Then $B \Kin$ is a tubular neighborhood of $\Kin$
in $\Gin$.  Considering the action by 
$\eta = \ch_{B\Kin} \otimes \bigotimes_{v<\infty}\ch_{\Kv}$, we have
\begin{equation*}\begin{split}
  (\eta\cdot f)(g)
  &=\int_{\Ga} \eta(h) \, f(gh) \dd h
  = \int_{\Ga} \eta(g^{-1}h) \, f(h) \dd h\\
  &=\int_{\Gq} \sum_{\ga\in\Gk} \eta(g^{-1}\ga h) \, f(h) \dd h
  = \<\eta_g, \cj f\>,
\end{split}\end{equation*}
where
\begin{equation*}
  \eta_g(h)
  = \sum_{\ga\in\Gk} \eta(g^{-1}\ga h).
\end{equation*}
If the radius $1/T$ is sufficiently small, this sum has one single term.
(Note that $\eta$ is not smooth; such a choice avoids cut-off functions.)  
Still following Garrett,
\begin{equation*}\begin{split}
  \Abs{\eta_g}^2
  &=\int_{\Gq} \eta_g(h) \; \sum_{\ga\in\Gk} \cj\eta(g^{-1}\ga h) \dd h
  = \int_{\Ga} \eta_g(h) \, \cj\eta(g^{-1}h) \dd h\\
  &=\int_{\Ga} \sum_{\ga\in\Gk} \eta(g^{-1}\ga gh) \, \cj\eta(h) \dd h.
\end{split}\end{equation*}
We are led to
\begin{equation*}
  \Abs{\eta_g}^2 \ll \text{radius}^n = \frac1{T^n}.
\end{equation*}

On the other hand, because $\Phi$ is right $K$--invariant and generates 
an irreducible representation, it must be that
\begin{equation*}
  \<\eta_g, \Phi\>
  = (\eta \cdot \cj\Phi)(g)
  = C \, \cj\Phi(g),
\end{equation*}
where the constant $C$ depends only on $\eta$ and the archimedean parameters 
of $\Phi$.  Let $s$ be the archimedean parameter, seen as the parameter of 
a principal series representation $\ph_s$ at $\infty$.  Then, if 
$n_a m_\la \in B$, $k \in \Kin$, and $s \ll T$ (which is the case if 
$\abs{\la_\Phi} < T^2$, as $\la_\Phi \asymp s^2$), we have
\begin{equation*}
  \ph_s(n_a m_\la k)
  = \de(m_\la)
  = \abs\la^{ns}
  < e^{ns/T}
  \ll 1.
\end{equation*}
In particular, assume that $g$ lies in a fixed 
compact and the radius $1/T$ is sufficiently small.  Then
\begin{equation*}
  (\eta \cdot \cj\Phi)(g)
  = \int_{\Ga} \eta(h) \, \cj\Phi(gh) \dd h
  \gg {\int_{\Ga} \eta(h) \dd h} \cdot \cj\Phi(g)
  \asymp \text{radius}^n \cdot \cj\Phi(g)
\end{equation*}
and we see that
\begin{equation*}
  C \gg \text{radius}^n = \frac1{T^n}.
\end{equation*}

Combining all this information, we conclude
\begin{equation*}
  \frac1{T^n}
  \gg \Abs{\eta_g}^2
  = \sint \abs*{\<\eta_g,\Phi\>}^2\dd\Phi
  \ge \sint_{\abs{\la_\Phi}<T^2} \abs*{\<\eta_g,\Phi\>}^2
  \gg \sint_{\abs{\la_\Phi}<T^2} \frac{\abs{\Phi(g)}^2}{T^{2n}}.\qedhere
\end{equation*}
\end{proof}

\begin{lem}\label{l:delta-ade}
Let $\de_\ade$ be the distribution defined, for right $K$--invariant $f$,
by
\begin{equation*}
  \int_{\Gq} f\cdot\de_\ade
  = f(1).
\end{equation*}
We have $\de_\ade\in \Hauto{-n/2-\ep}$, for any $\ep>0$.
\end{lem}

(As the definition of $\Hauto\ell$ depends on the \emph{global} spectral 
decomposition, it is not possible to reduce this to classical lemmas, of
which it is a direct analogue.)

\begin{proof}
Indeed, let
\begin{equation*}
  a_N = \sint_{\abs{\la_\Phi}<N^2}\abs*{F(1)}^2.
\end{equation*}
Then
\begin{equation*}\begin{split}
  \sint\frac{\abs*{\<\de_\ade,\Phi\>}^2}{(1+\abs{\la_\Phi})^{n/2+\ep}}\dd\Phi
  &\ll \sum_{N\ge0}\frac{a_{N+1}-a_N}{(1+N)^{n+2\ep}}\\
  &= \sum_{N\ge0}
     a_N\biggl(\frac1{N^{n+2\ep}}-\frac1{(N+1)^{n+2\ep}}\biggr)\\
  &\ll \sum_{N\ge0}\frac{a_N}{N^{n+1+\ep}}
  \ll \sum_{N\ge0}\frac{N^n}{N^{n+1+\ep}}
  < \infty.
\end{split}\end{equation*}

Therefore, in $\Hauto{-n/2-\ep}$, we have
\begin{equation*}
  \de_\ade
  = \sint \Phi(1) \cdot \Phi\dd\Phi.\qedhere
\end{equation*}
\end{proof}

\begin{prop}\label{p:auto-Sob:sol}
If $\abs{\la_\Phi - \la} \ge r > 0$ for all $\Phi$,
then $(\Cas - \la) : \Hauto\ell \to \Hauto{\ell-2}$ is an isomorphism.
\end{prop}

\begin{proof}
For $\la\in\CC$ and $f\in \Hauto\ell$, we have
\begin{equation*}\begin{split}
  \Abs{(\Cas-\la) f}_{\ell-2}^2
  &= \sint \abs{\<f, \Phi\> \; (\la_\Phi-\la)}^2
     \, (1 + \abs{\la_\Phi})^{\ell-2} \dd\Phi\\
  &\ll \sint\abs{\<f, \Phi\>}^2 \, (1 + \abs{\la_\Phi})^\ell \dd\Phi
  = \Abs{f}^2_\ell,
\end{split}\end{equation*}
showing that $(\Cas-\la) : \Hauto\ell \to \Hauto{\ell-2}$ is continuous and 
injective.
On the other hand, if $\abs{\la_\Phi-\la}\ge r>0$ for all $\Phi$, then
\begin{equation*}
  1+\abs{\la_\Phi}
  \le 1+\abs{\la_\Phi-\la}+\abs\la
  \ll \abs{\la_\Phi-\la}.\qedhere
\end{equation*}
\end{proof}

For example, there is a unique solution $u_\ade$ of
\begin{equation*}
  (\Cas-\la)^N u_\ade = \de_\ade,
\end{equation*}
for $\de_\ade$ as defined above.
In $\Hauto{2N-n/2-\ep}$, it can be expressed as
\begin{equation}
  u_\ade
  = \sint \frac{\Phi(1)}{(\la_\Phi-\la)^N}\cdot \Phi\dd\Phi.
\end{equation}

\section{Zonal spherical functions}\label{s:zonal-Sob}

In this section, we work at the single archimedean place (suppressed).
The facts we need on spherical functions were taken from the
monographs by Helgason \cite{He} and Gangolli and Varadarajan \cite{GaVa}.  
In this \emph{summary}, we follow mostly Helgason, as well as
Garrett \cite{Ga2}, with adaptations for the rank one case.

Let $X = \Gin/\Kin$, $n = \dim_\RR X$, and $\Lap$ be the image of the
Casimir element $\Cas$ on $X$.  A smooth function $f$ on $K \lqu G/K$ is a 
\emph{zonal spherical function} if it is an eigenfunction of $\Lap$ normalized 
by $f(1) = 1$.  For any given eigenvalue $\la$, there is only one such $f$.

Recall that we defined 
\begin{equation*}
  \ph_s(n_a m_\la k)
  = \de_P(m_\la)^s
  = \abs\la^s.
\end{equation*}
By a theorem of Harish-Chandra, all zonal spherical 
functions are of the form
\begin{equation*}
  \psi_s(g) = \int_K \ph_s(kg)\dd k,
\end{equation*}
for some $s\in\CC$.

The \emph{spherical transform} is defined for $f\in L^2(K\lqu G/K)$ by
\begin{equation*}
  \spht f(s)
  = \int_G f \cdot \psi_{1-s}.
\end{equation*}
The inversion formula (up to a constant) is
\begin{equation*}
  f(g)
  = \int_{\Re s=\frac12} \frac{\spht f(s)\cdot\psi_s(g)}{\abs{\HCc(s)}^2}\dd s,
\end{equation*}
with corresponding Plancherel identity
\begin{equation*}
  \Abs{f}^2_{L^2}
  = \int_{\Re s=\frac12} \frac{\abs{\spht f(s)}^2}{\abs{\HCc(s)}^2}\dd s.
\end{equation*}
(We need to use the Plancherel identity to establish an isometric extension.)
The Harish-Chandra function $\HCc(s)$ is given \cite{He} by the 
Gindikin--Karpelevi\v c formula, which, in our case, is
\begin{equation*}
  \HCc(s)
  = \frac{\Gam\Bigl(\bigl(s-\frac12\bigr)\frac{n-1}2\Bigr)\cdot
     \Gam\Bigl(\frac{3(n-1)}4\Bigr)}
    {\Gam\Bigl(\bigl(s+\frac12\bigr)\frac{n-1}2\Bigr)\cdot
     \Gam\Bigl(\frac{n-1}4\Bigr)}.
\end{equation*}
(Helgason's $i\la$ relates to our $s$ by $\rho + i\la = 2\rho s$.
The positive simple root  $\al$ has $\<\al, \al\> = n-1$ and multiplicity 
$(n - 1)$.  Therefore, $2\rho = (n - 1) \al$.)
The main fact we need is that
\begin{equation*}
  \HCc(s) \asymp \abs{s}^{-\frac{n-1}2}.
\end{equation*}

We define the zonal Sobolev norm by
\begin{equation*}
  \Abs{f}^2_\ell
  = \int_{\Re s=\frac12}
     \frac{\abs{\spht f(s)}^2}{\abs{\HCc(s)}^2}
     \, (1 + \abs{\la_s})^\ell \dd s,
\end{equation*}
where $\la_s = \la_{\psi_s} \asymp \abs s^2$.
We define the \emph{zonal spherical Sobolev space} as
\begin{equation*}
  \Hzon\ell
  = \,\text{closure of $\ScD(K\lqu G/K)$ with respect to $\Abs\;_\ell$}.
\end{equation*}

\begin{lem}
If $\de_\infty$ is the delta distribution centered at $1\cdot K$, we have
$\de_\infty\in \Hzon{-n/2-\ell}$, for any $\ep>0$.
\end{lem}
(This is compatible with the outcome for $\de_\ade$, in 
lemma \thref{l:delta-ade}.)

\begin{proof}
We have
\begin{equation*}
  \spht{\de_\infty}(s)
  = \psi_{1-s}(1)
  = \int_K\ph_{1-s}(k)\dd k
  = 1.
\end{equation*}
On the other hand, 
\begin{equation*}
  \frac{(1+\abs{\la_s})^{-\ell}}{\abs{\HCc(s)}^2}
  \asymp \frac{\abs{s}^{-2\ell}}{\abs{s}^{-(n-1)}}
\end{equation*}
and the requirement for
\begin{equation*}
  \int_{\Re s=\frac12}
  \frac{(1+\abs{\la_s})^{-\ell}}{\abs{\HCc(s)}^2}\dd s < \infty
\end{equation*}
is $2\ell - (n-1) > 1$, or $\ell > n / 2$.
\end{proof}

The spherical expansion of $\de_\infty$, valid in $\Hzon{-n/2-\ep}$, is
\begin{equation*}
  \de_\infty
  = \int_{\Re s=\frac12}\frac{\psi_s}{\abs{\HCc(s)}^2}\dd s.
\end{equation*}

Exactly as in proposition \thref{p:auto-Sob:sol}, 
$(\Lap-\la) : \Hzon\ell \to \Hzon{\ell-2}$ is an isomorphism
provided $\la$ is away from all eigenvalues of $\Lap$ (with $s$ lying
on $\Re s=\frac12$ this is not at all an issue).
In that case, also as before, there is a solution $u_\infty$ of 
$(\Lap-\la)^N u_\infty=\de_\infty$. 
In $\Hzon{2N-n/2-\ep}$,
\begin{equation*}
  u_\infty
  = \int_{\Re s=\frac12}\frac{\psi_s}{(\la_s-\la)^N\cdot\abs{\HCc(s)}^2}\dd s.
\end{equation*}

\section{Poincar\'e series and spectral identities}\label{s:Pe}

We return to the global picture and follow Diaconu and Garrett 
\cites{DiGa1, DiGa2, Ga2, Ga25}.

At non-archimedean $v$, let $u_v$ be the characteristic function of 
$\Hv\cdot\Kv$ and $u=u_\infty\otimes\bigotimes_{v<\infty}u_v$.
Noting that $u_\infty$ inherits the left $\Hin$--invariance of 
$\de_\infty$, we define the \emph{Poincar\'e series}
\begin{equation*}
  \Pe(g)
  = \sum_{\ga\in\Hk\lqu\Gk} u(\ga g).
\end{equation*}
(This is a function on $\Gq$, while $u_\infty$ is a function on $\Gin$.)

For brevity, write $P(\Cas)=(\Cas-\la)^N$.  Require $P(\la_\Phi)\gg0$.
Note that $\Pe$ is left $\Gk$--invariant and that $\Cas$ acts only on the
archimedean information. 
It is clear that the Poincar\'e series is a solution of
\begin{equation*}
  P(\Cas) \, \Pe = \de_\ade.
\end{equation*}
This same solution was shown in section~\ref{s:auto-Sob} to be unique
in the automorphic Sobolev space. Therefore, it must be that
\begin{equation*}
  \sum_{\ga\in\Hk\lqu\Gk}u(\ga g)
  = \Pe
  = \sint \frac{\Phi(1)}{P(\la_\Phi)}\cdot \Phi\dd\Phi.
\end{equation*}
(Unremarkably, $\Pe$ has a larger support than $\de_\ade$. The same phenomenon 
occurs already with fundamental solutions of $\Lap$ in $\RR^n$, whose support 
is all of $\RR^n$, while the support of $\de$ is only $\{0\}$.)

On the other hand, 
let $f:\Gq\to\CC$ be an eigenfunction of $\Cas$ with eigenvalue 
$\la_f\ne\la$.  Then
\begin{equation*}\begin{split}
  \<f,\Pe\>_G
  &= \int_{\Gq}f\cdot\cj\Pe
  = \int_{\Ha\lqu\Ga} \int_{\Hq}f\cdot\cj u\\
  &= \int_{\Ha\lqu\Ga}
      \Bigl(\int_{\Hq}\frac{P(\Cas)}{P(\la_f)}\,f\Bigr)\cdot\cj u
  = \int_{\Ha\lqu\Ga}
      \int_{\Hq}f\cdot\frac{P(\Cas)}{P(\la_f)}\,\cj u\\
  &= \frac1{P(\la_f)} \int_{\Ha\lqu\Ga} \int_{\Hq}f\cdot
      \Bigl(\de_\infty\otimes\bigotimes_{v<\infty}u_v\Bigr)
  = \frac1{P(\la_f)} \int_{\Hq} f\\
  &= \frac{(f)_H}{P(\la_f)}.
\end{split}\end{equation*}
That is, the Poincar\'e series can be used to extract information
about periods.

In $\Hauto{2N-n/2-\ep}$, we decompose
\begin{equation*}
  \Pe
  = \sint \<\Pe, \Phi\> \cdot \Phi \dd\Phi
  = \sint \frac{(\cj\Phi)_H}{P(\la_{\cj\Phi})} \cdot \Phi \dd\Phi
  = \sint \frac{(\Phi)_H}{P(\la_\Phi)} \cdot \cj\Phi \dd\Phi.
\end{equation*}
Diaconu and Garrett \cite{DiGa1} discuss a similar decomposition 
for $G=\GL_2$.

\subsection*{Application to spectral identities}

Still following Diaconu and Garrett \cite{DiGa1}, consider two chains of 
inclusions: $H^\Delta\subseteq G^\Delta\subseteq G\times G$
and $H^\Delta\subseteq H\times H\subseteq G\times G$,
where $G^\Delta$ denotes the image of $G\to G\times G:g\mapsto(g,g)$, and 
similarly for $H^\Delta$.

Let $f\otimes f'$ be an automorphic function on $G\times G$. 
The two inclusions suggest two different evaluations of
\begin{equation*}
  \<f \cdot f', \Pe\>_G
  = \int_{\Hk\lqu\Ga} f \cdot f'\cdot u;
\end{equation*}
a spectral decomposition along $G^\Delta$ (we will call it the \emph{spectral
expansion}) or along $H\times H$ (we will call it the \emph{moment expansion}).

Decomposing along $G$ (the spectral expansion), we have
\begin{equation}\label{e:Pe:spec}
  \<f\cdot f', \Pe\>_G
  = \sint_{\text{$\Phi$ on $G$}} \<f\cdot f', \Phi\> \, \<\Phi, \Pe\>
  = \sint \frac{(\Phi)_H}{P(\la_\Phi)} \;
    \int_{\Gq} f \cdot f' \cdot \cj\Phi \dd\Phi,
\end{equation}
involving triple products as well as the periods $(\Phi)_H$ of each
component $\Phi$.

Note that $f$ has a discrete decomposition along 
$H$.  Writing $(g \cdot f)(h)=f(hg)$:
\begin{equation*}
  g\cdot f
  = \sum_F (g\cdot f,F)_H\cdot F,
\end{equation*}
with $F$ running over an orthonormal basis of eigenfunctions of
$\scr Z(\frk h_\infty)$.  We obtain the moment expansion:
\begin{equation}\label{e:Pe:mom}\begin{split}
  \<f \cdot f', \Pe\>_G
  &= \int_{\Hk\lqu\Ga} f \cdot f' \cdot u
  = \int_{\Ha\lqu\Ga} \int_{\Hq} f(hg) \, f'(hg) \, u(g) \dd h \dd g\\
  &= \int_{\Ha\lqu\Ga} (g \cdot f, \smash{\cj{g \cdot f'}})_H \; u(g) \dd g\\
  &= \int_{\Ha\lqu\Ga}\sum_F (g \cdot f, F)_H \, (g \cdot f',\cj F)_H
      \; u(g)\dd g.
\end{split}\end{equation}
Often, it can be rewritten in the form
\begin{equation*}
  \sum_F (f, F)_H \, (f', \cj F)_H \cdot
    \text{weight}(f_\infty, f'_\infty, F_\infty),
\end{equation*}
with the weight depending only on the archimedean parameters.  In the
next section, we show the details of such a rewriting when $f$ and $f'$ are
spherical Eisenstein series, unramified at non-archimedean places.
For applications, one would need to study its asymptotics.

In sum, we establish:

\begin{thm}
Let $f \otimes f'$ be an automorphic function on $G \times G$, where $f$ and
$f'$ are spherical Eisenstein series, unramified at non-archimedean places.  Then 
$\<f \cdot f', \Pe\>_G$ has two expansions: the spectral expansion
\begin{equation*}
  \sint_\Phi\frac{(\Phi)_H}{P(\la_\Phi)} \,
     \int_{\Gq} f \cdot f' \cdot \cj\Phi \dd\Phi
\end{equation*}
is a decomposition along $G$, while the moment expansion
\begin{multline*}
  \sum_F \int_{\Ha\lqu\Ga}
     (g \cdot f, F)_H \, (g \cdot f', \cj F)_H \; u(g) \dd g\\
  = \sum_F (f, F)_H \, (f', \cj F)_H \cdot
    \textup{weight}(f_\infty, f'_\infty, F_\infty)
\end{multline*}
is a decomposition along $H$.
\end{thm}

\section{Eisenstein series and their second moments}\label{s:Eis}

We want to specialize to spherical, unramified, Eisenstein series 
$f = E_a$ and $f' = E_b$.  Here, $a, b \in \CC$, and $E_a$ and $E_b$ are
parametrized as discussed in the introduction.

One first obstacle is that 
$f \cdot f'$ is not in $L^2(\Gq)$ and it is unclear whether we can integrate
\begin{equation*}
  \<f \cdot f', \Pe\>_G
  = \int_{\Gq} f \cdot f' \cdot \Pe
\end{equation*}
directly.  It is possible to subtract finitely many singular terms from
$f \cdot f'$ so that the difference is square-integrable; we discuss that in 
appendix \ref{s:regularize}.

The exact choice of singular terms will depend on where $a$ or $b$ 
lie.  For definiteness, say
\begin{equation*}
  \cal F
  = E_a \, E_b + \sum_s c_s \, E_s.
\end{equation*}
(with finitely many $s$ occurring) is the regularized expression.

For the spectral expansion, we have, as in \eqref{e:Pe:spec},
\begin{equation*}
  \<\cal F, \Pe\>_G
  = \sint_{\text{$\Phi$ on $G$}} \<\cal F, \Phi\> \, \<\Phi, \Pe\>
  = \sint \<\cal F, \Phi\> \, \frac{(\Phi)_H}{P(\la_\Phi)} \dd\Phi.
\end{equation*}

The moment expansion starts as \eqref{e:Pe:mom},
\begin{equation*}
  \<\cal F,\Pe\>_G
  = \int_{\Gq}\cal F\cdot \Pe
  = \int_{\Hk\lqu\Ga}\cal F\cdot u
  = \int_{\Ha\lqu\Ga}\int_{\Hq}\cal F(hg)\dd h\cdot\,u(g)\dd g,
\end{equation*}
where the convergence of the inner integral is justified by
the compactness of $\Hq$.
Recall that at non-archimedean $v$ we chose $u_v=\ch_{\Hv\cdot\Kv}$,
so we assume $g_v\in\Hv\cdot\Kv$.  Therefore, we can simplify
further:
\begin{equation}\label{e:Eis:mom-unw}
  \<\cal F, \Pe\>_G
  = \int_{\Hin\lqu\Gin} \int_{\Hq} \cal F(hg) \dd h \; u_\infty(g) \dd g.
\end{equation}
The inner integral is
\begin{equation}\label{e:Eis:inner}
  \int_{\Hq} \cal F(hg)\dd h
  = \int_{\Hq} E_a(hg) \, E_b(hg) \dd h
    + \sum_s c_s \, \int_{\Hq}E_s(hg) \dd h.
\end{equation}

\subsection*{The ``main'' part}

For the $E_a \, E_b$ summand, we have, as before,
\begin{equation}\label{e:Eis:main-dbl}
  \int_{\Hq} E_a(hg) \, E_b(hg) \dd h
  = \sum_F (g \cdot E_a, F)_H \, (g \cdot E_b, \cj F)_H.
\end{equation}

We remark that
\begin{equation*}\begin{split}
  (g \cdot E_s, F)_H
  &= \int_{\Hq} E_s(hg) \, \cj F(h) \dd h
  = \int_{\Sk\lqu\Ha} \ph_s(hg) \, \cj F(h) \dd h\\
  &= \int_{\Sa\lqu\Ha} \ph_s(hg) \, \int_{\Sq} \cj F(\theta h) \dd\theta \dd h.
\end{split}\end{equation*}
The function
\begin{equation*}
  F_\Th(h)
  = \int_{\Sq} F(\theta h) \dd\theta
\end{equation*}
is a spherical vector in $\Ind_\Th^H1$, normalized by $F_\Th(1) = (F)_\Th$.
Therefore, with $\eta$ a spherical vector normalized by $\eta(1) = 1$, 
we obtain
\begin{equation*}
  (g \cdot E_s, F)_H
  = \int_{\Sa\lqu\Ha} \ph_s(hg) \, \cj F_\Th(h) \dd h.
  = (\cj F)_\Th \cdot \int_{\Sa\lqu\Ha} \ph_s(hg) \, \cj\eta(h) \dd h.
\end{equation*}

Recalling that $g_v\in\Hv\cdot\Kv$ for non-archimedean $v$,
we see that all but the archimedean factor are independent of $g$ and
\begin{equation*}
  (g\cdot E_s,F)_H
  = \Bigl(\int\limits_{\Sin\lqu\Hin}
          \ph_{s,\infty} (hg_\infty) \, \cj F_\Th(h) \dd h
    \Bigr) \cdot \prod_{v<\infty}
    \Bigl(\int\limits_{\Sv\lqu\Hv} \ph_{s,v}(h) \, \cj F_\Th(h) \dd h\Bigr).
\end{equation*}
We abbreviate this as follows:
\begin{align*}
  \psi_{s, F}(g_\infty)
  &= \int_{\Sin\lqu\Hin} \ph_{s,\infty}(hg_\infty) \, \cj F_\Th(h) \dd h;\\
  (E_s,F)'_H
  &= \prod_{v<\infty}\int_{\Sv\lqu\Hv} \ph_{s,v}(h) \, \cj F_\Th(h) \dd h;\\
  (g \cdot E_s, F)_H
  &= \psi_{s,F}(g_\infty) \, (E_s,F)'_H.
\end{align*}

Combining this with \eqref{e:Eis:main-dbl}, we see that
the ``main'' part of the moment expansion \eqref{e:Eis:mom-unw} is
\begin{equation}\label{e:Eis:main-trpl}
  \sum_F (E_a,F)'_H \, (E_b,\cj F)'_H \, \int_{\Hin\lqu\Gin}
    \psi_{a,F}(g) \, \psi_{b,\cj F}(g) \, u_\infty(g) \dd g.
\end{equation}
Suppress for a moment the
$\infty$ indices, and use $G=HP$, with measure $\dd{(hp)}=\dd h\dd p$ and
$\dd p$ being a \emph{right} Haar measure.  We have
\begin{equation*}
  \int_{H\lqu G} \psi_{a,F} \cdot \psi_{b,\cj F} \cdot u
  = \int_H\int_H\int_P \ph_a(hp) \, \cj F_\Th(h) \cdot
    \ph_b(h'p) \, F_\Th(h') \cdot u(p) \dd p \dd h \dd{h'}.
\end{equation*}
With a nod to Diaconu and Garrett \cite{DiGa1}, set
\begin{equation*}
  X_{a,b}(h,h')
  = \int_P \ph_a(hp) \, \ph_b(h'p) \, u_\infty(p)\dd p
\end{equation*}
and conclude
\begin{equation*}
  \int_{H\lqu G}\psi_{a,F} \cdot \psi_{b,\cj F} \cdot u
  = \int_{\Th\lqu H}\int_{\Th\lqu H}
    \cj F_\Th(h) \, F_\Th(h') \, X_{a,b}(h,h') \dd h \dd{h'}.
\end{equation*}

Resuming \eqref{e:Eis:main-trpl}, we see that the ``main'' part
of the moment expansion is
\begin{equation}\label{e:Eis:main}
  \sum_F (E_a,F)'_H \, (E_b,\cj F)'_H \,
    \int_{\Sin\lqu\Hin}\int_{\Sin\lqu\Hin}
    \cj F_\Th(h) \, F_\Th(h') \, X_{a,b}(h,h') \dd h \dd{h'}.
\end{equation}
Recalling that $F_\Th(h) = (F)_\Th \, \eta_F(h)$, where $\eta_F$ is a spherical
vector in $\Ind_\Th^H 1$ normalized by $\eta_F(1) = 1$, we can make the periods
even more apparent.

\subsection*{The ``singular'' part}

For the other summands in \eqref{e:Eis:inner}, we observe that, by 
Witt's lemma, $\Pk\lqu\Gk$ is the 
space of isotropic lines in $k^{n+1}$, on which $\Hk$ acts transitively.  As
$\Th = H \cap P$, we have $\Pk\lqu\Gk = \Sk\lqu\Hk$ and
\begin{equation*}\begin{split}
  \int_{\Hq} E_s(hg) \dd h
  &= \int_{\Hq} \sum_{\ga \in \Sk\lqu\Hk} \ph_s(\ga hg) \dd h
  = \int_{\Sk\lqu\Ha} \ph_s(hg) \dd h\\
  &= \vol(\Sa) \int_{\Sa\lqu\Ha} \ph_s(hg) \dd h.
\end{split}\end{equation*}
Normalizing $\vol(\Sa) = 1$ and recalling that $g_v \in \Hv\cdot\Kv$ for 
non-archimedean $v$, we obtain
\begin{equation*}\begin{split}
  \int_{\Hq} E_s(hg) \dd h
  &= \Bigl(\int_{\Sin\lqu\Hin} \ph_{s,\infty}(hg_\infty) \dd h\Bigr)
     \cdot\prod_{v<\infty}
     \Bigl(\int_{\Sv\lqu\Hv} \ph_{s,v}(h) \dd h\Bigr)\\
  &= \psi_s(g) \, (E_s)_H.
\end{split}\end{equation*}
Additionally, because $u_\infty$ is a solution of 
$P(\Lap)u_\infty=(\Lap-\la)^N u_\infty=\de_\infty$, we have
\begin{equation*}\begin{split}
  \int_{\Hin\lqu\Gin}\psi_s\cdot u_\infty
  &=\int_{\Hin\lqu\Gin}\frac{P(\Lap)\psi_s}{P(\la_s)}\cdot u_\infty\\
  &=\int_{\Hin\lqu\Gin}\psi_s\cdot\frac{P(\Lap)u_\infty}{P(\la_s)}
  = \int_{\Hin\lqu\Gin}\psi_s\cdot\frac{\de_\infty}{P(\la_s)}
  = \frac{1}{P(\la_s)}.
\end{split}\end{equation*}

Therefore, the ``singular'' part of the moment 
expansion \eqref{e:Eis:mom-unw} becomes 
\begin{equation}\label{e:Eis:sing}
  \int_{\Hin\lqu\Gin}\sum_s c_s \, (g \cdot E_s, 1)_H
  = \sum_s c_s \, (E_s)_H \,
    \int_{\Hin\lqu\Gin}\psi_s \cdot u_\infty
  = \sum_s c_s \, \frac{(E_s)_H}{P(\la_s)}.
\end{equation}

Combining this with the ``main'' part \eqref{e:Eis:main}, we obtain
the complete moment expansion:

\begin{prop}
Let
\begin{equation*}
  \cal F
  = E_a \, E_b + \sum_s c_s \, E_s
\end{equation*}
be the regularized expression (with finitely many $s$ occurring) for 
$E_a \, E_b$, and
\begin{gather*}
  X_{a,b}(h, h')
  = \int_{\Pin}
     \ph_{a,\infty}(hp) \, \ph_{b,\infty}(h'p) \, u_\infty(p) \dd p,\\
  \textup{weight}_{a,b,F}
  = \abs{(F)_\Th}^2 \int_{\Hin}\int_{\Hin}
    \cj \eta_F(h) \, \eta_F(h') \, X_{a,b}(h,h')\dd h\dd{h'},
\end{gather*}
where $\eta_F$ is a spherical vector in $\Ind_\Th^H 1$ normalized by 
$\eta_F(1) = 1$.
Then the moment expansion of $\<\cal F, \Pe\>_G$ is
\begin{equation*}
  \<\cal F, \Pe\>_G
  = \sum_F (E_a,F)'_H \, (E_b,\cj F)'_H \cdot \textup{weight}_{a,b,F}
    + \sum_s c_s \, \frac{(E_s)_H}{P(\la_s)}.
\end{equation*}
\end{prop}

The actual computation of $X_{a,b}$ and $\textup{weight}_{a,b,F}$ can get quite
involved, as illustrated, for example, in the $GL(r) \times \GL(r-1)$ case
discussed by Diaconu, Garrett and Goldfeld \cite{DiGaGo}.

\appendix
\section{Regularizing functions not of rapid decay}\label{s:regularize}

In the previous section, we needed the spectral expansion of $E_a E_b$, and
observed that one difficulty was that that product is not in $L^2(\Gq/K)$.
However, it is possible to subtract a linear combination of Eisenstein series
(the singular part), so that the difference is an $L^2$ function.

The idea, which I learned from Garrett \cites{Ga2, Ga25} and he traces to 
Zagier \cite{Za}, uses the constant terms of the Eisenstein series to guide 
the choice of singular terms, so as to assure cancellation of non-$L^2$ terms.  
We articulate the details in our specific case $G=O(n,1)$.

We saw in the introduction that $M = \Th A$, where $A = \{m_\la\} \cong \GL(1)$.
We will always write the elements of $P = N\Th A$ in the form 
$p = n\theta m_\la$.  Because $\dd n\dd\theta\dd{(m_\la)}$ is a right 
invariant measure, $\dd p = \de_P(m_\la)^{-1} \cdot \dd n\dd\theta\dd{(m_\la)}$ 
is a \emph{left} invariant measure on $P$.  In the same manner, we always write
the elements of $G = PK = N\Th AK$ in the form $g = pk = n\theta m_\la k$, in 
which case 
$\dd g = \dd p \dd k = \de_P(m_\la)^{-1}\cdot\dd n\dd\theta\dd{(m_\la)}\dd k$ 
is a Haar measure on $G$.

Recall now that we can choose a compact $C \subset \Na\Sa$ and a real 
$t_0 > 0$ such that the Siegel set
\begin{equation*}
  \frk S
  = \{g = n\theta m_\la k : n\theta \in C \text{ and } \de_P(m_\la) \ge t_0\}
\end{equation*}
satisfies $\Gk\frk S = \Ga$. We assume such a choice was made.

Supposing
\begin{equation*}
  f(g) \ll \de_P(m_\la)^\sigma
\end{equation*}
for some real $\sigma$, we have
\begin{equation*}\begin{split}
  \int_{\Gq} f
  &\le \int_{\frk S} f
  \ll \int_K \int_{t_0}^\infty \int_C 
      \abs{f(n\theta m_\la k)}
      \cdot \de_P(m_\la)^{-1} \cdot \dd{(n\theta)}\frac{\dd\la}\la \dd k\\
  &\ll \int_{t_0}^\infty \de_P(m_\la)^{\sigma - 1} \frac{\dd\la}\la
  = \int_{t_0}^\infty \abs\la^{n(\sigma - 1) - 1} \dd\la 
\end{split}\end{equation*}
(in the last step, we used $\de_P(m_\la) = \abs\la^n$). This last integral 
converges when $\sigma < 1$. We have thus shown that $f$ is integrable over 
$\Gq$ provided $\sigma < 1$. For $L^2$ integrability, we need 
$\sigma < \frac12$.

Recall next that a function $f$ on $\Pk\lqu\Ga$ is of moderate growth if
\begin{equation*}
  f(g)
  \ll \de_P(m_\la)^\sigma
  \qquad\text{for some $\sigma > 0$}
\end{equation*}
and of rapid decay if 
\begin{equation*}
  f(g)
  \ll \de_P(m_\la)^\sigma
  \qquad\text{for all $\sigma < 0$}.
\end{equation*}
From the discussion above, it is apparent that if $f$ is right $\Gk$--invariant 
and of rapid decay, then it is integrable over $\Gq$.

We also know \cites{Lan, MoWa} that, choosing the normalization $\vol(\Nq) = 1$,
the constant term of the Eisenstein series is
\begin{equation*}
  \const E_s(g)
  = \de_P(m_\la)^s + c_s \cdot \de_P(m_\la)^{1-s},
\end{equation*}
where $c_s$ is the same constant as in the functional equation
\begin{equation*}
  E_{1-s} = c_{1-s} \cdot E_s.
\end{equation*}
Moreover, it is a standard fact that $f - \const f$ is of rapid decay, so
we can write
\begin{equation*}
  E_s(g)
  = \de_P(m_\la)^s + c_s \cdot \de_P(m_\la)^{1-s} + \text{fn rapid decay}.
\end{equation*}

We return to the case $E_a\cdot E_b$ with $a, b \in \CC$. Clearly,
\begin{multline*}
  E_a(g) \cdot E_b(g)
  = \de_P(m_\la)^{a + b}
     + c_a \cdot \de_P(m_\la)^{1 - a + b}\\
     + c_b \cdot \de_P(m_\la)^{a + 1 - b}
     + c_a \cdot c_b \cdot \de_P(m_\la)^{1 - a + 1 - b}
     + \text{fn rapid decay}.
\end{multline*}

As we know that exponents less than $\frac12$ assure $L^2$ integrability, we 
usually can say more.

For example, if $\Re a > 1$ and $\Re b = \frac12$,
\begin{equation*}
  E_a(g) \cdot E_b(g)
  = \de_P(m_\la)^{a + b} + c_b \cdot \de_P(m_\la)^{a + 1 - b}
     + \text{$L^2$ function}.
\end{equation*}
Moreover,
\begin{equation*}\begin{split}
  E_{a+b}(g)
  &= \de_P(m_\la)^{a+b} + c_{a+b} \cdot \de_P(m_\la)^{1-a-b}
     + \text{fn rapid decay}\\
  &= \de_P(m_\la)^{a+b} + \text{$L^2$ function}.
\end{split}\end{equation*}
In the same manner,
\begin{equation*}\begin{split}
  E_{a+1-b}(g)
  &= \de_P(m_\la)^{a+1-b} + c_{a+1-b} \cdot \de_P(m_\la)^{-a+b}
     + \text{fn rapid decay}\\
  &= \de_P(m_\la)^{a+1-b} + \text{$L^2$ function}.
\end{split}\end{equation*}
Therefore,
\begin{equation*}
  E_a\cdot E_b - E_{a+b} - c_b \cdot E_{a+1-b}
  = \text{$L^2$ function}.
\end{equation*}

We may well have more than two singular terms.  For example, if 
$\Re a = \Re b = \frac12$, we obtain:
\begin{multline*}
  E_a(g) \cdot E_b(g)
  = \de_P(m_\la)^{a + b}
     + c_a \cdot \de_P(m_\la)^{1 - a + b}\\
     + c_b \cdot \de_P(m_\la)^{a + 1 - b}
     + c_a \cdot c_b \cdot \de_P(m_\la)^{2 - a - b}
     + \text{$L^2$ function}.
\end{multline*}
Here all exponents have real part equal to $1$.  But the important point is 
that if one exponent in
\begin{equation*}
  E_s(g)
  = \de_P(m_\la)^s + c_s \cdot \de_P(m_\la)^{1-s} + \text{$L^2$ function}
\end{equation*}
has real part greater than $\frac12$, the other one will have it less than
$\frac12$.  In our case, we have
\begin{align*}
  E_{a+b}(g)
  &= \de_P(m_\la)^{a+b} + \text{$L^2$ function};\\
  E_{1-a+b}(g)
  &= \de_P(m_\la)^{1-a+b} + \text{$L^2$ function};\\
  E_{a+1-b}(g)
  &= \de_P(m_\la)^{a+1-b} + \text{$L^2$ function};\\
  E_{2-a-b}(g)
  &= \de_P(m_\la)^{2-a-b} + \text{$L^2$ function}.
\end{align*}
Therefore,
\begin{equation*}
  E_a\cdot E_b - E_{a+b} - c_a \cdot E_{1-a+b}
    - c_b \cdot E_{a+1-b} - c_a\cdot c_b\cdot E_{2-a-b}
  = \text{$L^2$ function}.
\end{equation*}

\end{document}